\documentclass[12pt,reqno]{amsart}
\usepackage{amsmath,amsthm,amsfonts,amssymb}
\usepackage{graphicx,enumerate}
\usepackage{color}
\usepackage{currvita}
\usepackage{fancyhdr}
\usepackage{tikz}
\usepackage[T1]{fontenc}

\usepackage{hyperref}
\hypersetup{
  colorlinks=true,
  linkcolor=magenta,
  citecolor=blue,
}

\newtheorem{thm}{Theorem}[section]
\newtheorem{lem}[thm]{Lemma}
\newtheorem{cor}[thm]{Corollary}
\theoremstyle{definition}
\newtheorem{dfn}[thm]{Definition}
\newtheorem{ex}[thm]{Counterexample}
\newtheorem{cl}[thm]{Algorithm}

\DeclareMathOperator{\diag}{diag}

\newcommand{\T}{ T}
\renewcommand{\P}{\mathbb{P}}
\newcommand{\D}{\mathbb{D}}
\newcommand{\M}{\mathbb{M}}
\newcommand{\ba}{\boldsymbol{a}}
\newcommand{\bb}{\boldsymbol{b}}
\newcommand{\allone}{\mathbf{1}}
\newcommand{\allz}{\mathbf{0}}

\title{Transforming a matrix into a standard form}
\author{Akihiro Munemasa}
\address{Research Center for Pure and Applied Mathematics\\
Graduate School of Information Sciences\\
Tohoku University}
\email{munemasa@math.is.tohoku.ac.jp}

\author{Pritta Etriana Putri}
\address{Research Center for Pure and Applied Mathematics\\
Graduate School of Information Sciences\\
Tohoku University}
\email{pritta@ims.is.tohoku.ac.jp}
\begin{document}
\date{\today}
\maketitle

\begin{abstract}

We show that every matrix all of whose entries are in a fixed subgroup of the group of units of a commutative ring with identity is equivalent to a standard form. As a consequence, we improve the proof of Theorem 5 in D.~Best, H.~Kharaghani, H.~Ramp [Disc. Math. 313 (2013), 855--864].
\end{abstract}

\section{Introduction}
Throughout this note, we let $R$ be a commutative ring with
identity. We fix a subgroup $T$ of the group of units of $R$, and
set $T_0=T\cup\{0\}$. The set of $m\times n$ matrices with entries
in $T_0$ is denoted by $T_0^{m\times n}$. If $T=\{z\in\mathbb{C}:|z|=1\}$, then $W\in T_0^{n\times n }$ is called \emph{a unit weighing matrix of order $n$ with weight $w$} provided that $WW^*=wI$ where $W^*$ is the transpose conjugate of $W$. 
 Unit weighing matrices are introduced by D.~Best, H.~Kharaghani, and H.~Ramp in \cite{DB1,DB2}. Moreover, a unit weighing matrix is known as a unit Hadamard matrix if $w=n$ (see \cite{DB3}). A unit weighing matrix in which every entry is in $\{0,\pm1\}$ is called a \emph{weighing matrix}. We refer the reader to \cite{RH1} for an extensive discussion of weighing matrices, and to \cite{C1} for more information on applications of weighing matrices. 

The study on the number of inequivalent unit weighing matrices was initiated in \cite{DB1}. Also, observing the number of weighing matrices in 
standard form leads to an upper bound on the number of inequivalent unit weighing matrices \cite{DB1}. In this work, we will introduce a standard form of an arbitrary matrix in $T_0^{m\times n}$ and show that every matrix in $T_0^{m\times n}$ is equivalent to a matrix in standard form. 

We equip $T_0$ with a total ordering $\prec$
satisfying 
$
\min (T_0)=1\text{ and }\max(T_0)=0.
$
Moreover, let $\ba=(a_1, \ldots, a_n)$ and $\bb=(b_1, \ldots, b_n)$ be arbitrary row vectors with entries in $T_0$. 
If $k$ is the smallest index such that $a_k\neq b_k$, then we write $\ba<\bb$ 
provided $a_k\prec b_k$. We write $\ba\leq\bb$ if $\ba<\bb$ or $\ba=\bb$. 
If $\ba_1, \ldots, \ba_m$ are row vectors of a matrix $A\in T_0^{m\times n}$ 
and $\ba_1<\cdots<\ba_m$, then we say that the rows of $A$ are in \emph{lexicographical order}. 

\begin{dfn}\label{def:1}
We say that a matrix in $T_0^{m\times n}$
is in \emph{standard form} if the following conditions are satisfied:
\begin{itemize}
\item[(S1)] The first non-zero entry in each row is $1$.
\item[(S2)] The first non-zero entry in each column is $1$.
\item[(S3)] The first row is ones followed by zeros.
\item[(S4)] The rows are in lexicographical order according to $\prec$.
\end{itemize}
\end{dfn}

The subset of $T_0^{m\times m}$
consisting of permutation matrices, nonsingular diagonal matrices and
monomial matrices, are denoted respectively, by 
$\P_m,\D_m$ and $\M_m$. 
Then $\M_m=\P_m\D_m$.
\begin{dfn}

For $A,B\in T_0^{m\times n}$, we say that $A$ is \emph{equivalent} to $B$ if there exist monomial $T_0$-matrices $M_1$ and $M_2$ such that $M_1AM_2=B$.
\end{dfn}

We will restate the proof of \cite[Theorem 5]{DB1} as the following algorithm. 

\begin{cl}\label{c1}
Let $W$ be an arbitrary unit weighing matrix. 
\begin{enumerate}[(1)]
\item We multiply each $i$th row of $W$ by $r_i^{-1}$ where $r_i$ is the first non-zero entry in $i$th row. Denote the obtained matrix by $W^{(1)}$.
\item Let $c_j$ be the first non-zero entry in $j$th column of $W^{(1)}$. Let $W^{(2)}$ obtained from $W^{(1)}$ by multiplying each $j$th column by $c_j^{-1}$. 
\item Permute the columns of $W^{(2)}$ so that the first row has $w$ ones. Denote the resulting matrix by $W^{(3)}$.
\item Let $W^{(4)}$ be a matrix obtained from $W^{(3)}$ by sorting the rows of $W^{(3)}$ lexicographically with the ordering $\prec$. \end{enumerate}
Then $W^{(4)}$ is in standard form. 
\end{cl}
 
The steps (1)--(4) in Algorithm \ref{c1} was used in order to prove Theorem 5 in \cite{DB1}. However, we provide a counterexample to show that this algorithm does not produce a standard form.

\begin{ex}
The matrix
\[
W=
\begin{bmatrix}
1&-i&i&1&0&0\\
0&1&1&0&i&i\\
1&0&0&-1&-i&i\\
1&0&0&-1&i&-i\\
0&1&1&0&-i&-i\\
1&i&-i&1&0&0
\end{bmatrix}
\]
is a unit weighing matrix, where $i$ is a $4$th root of unity in $\mathbb{C}$. 
Also,we equip the set $\{0,\pm i,\pm 1\}$ with a total ordering $\prec$ 
defined by $1\prec -1\prec i \prec -i \prec 0$. 
Since the first nonzero entry in each row of $W$ is one, $W^{(1)}=W$. 
Applying step $(2)$, we obtain \[W^{(2)}=
\begin{bmatrix}
1&1&1&1&0&0\\
0&i&-i&0&1&1\\
1&0&0&-1&-1&1\\
1&0&0&-1&1&-1\\
0&i&-i&0&-1&-1\\
1&-1&-1&1&0&0
\end{bmatrix}
.\]
Notice that the first row of $W^{(2)}$ is all ones followed by zeros. So, $W^{(3)}=W^{(2)}$. 
Finally, by applying the last step of the algorithm, we have
\[
W^{(4)}=
\begin{bmatrix}
1&1&1&1&0&0\\
1&-1&-1&1&0&0\\
1&0&0&-1&1&-1\\
1&0&0&-1&-1&1\\
0&i&-i&0&1&1\\
0&i&-i&0&-1&-1\\
\end{bmatrix}
.\]

We see that $W^{(4)}$ is not in standard form. So, we conclude that the algorithm does not produce a matrix in standard form as claimed.
\end{ex}

This counterexample shows that the additional steps are needed to complete the proof of Theorem 5 in \cite{DB1}. In the next section, we will prove a more general theorem than  \cite[Theorem 5]{DB1} by showing that every matrix in $T_0^{m\times n}$ is equivalent to a matrix that is in 
standard form. 

\section{Main Theorem}

In addition to the conditions (S1)--(S4) in Definition \ref{def:1}, we will consider the following condition:
\begin{itemize}
\item[(S3)$'$] The first nonzero row is ones followed by zeros.
\end{itemize}

Note that (S3)$'$ is weaker than (S3). The condition (S3)$'$ is crucial in the proof of Lemma \ref{lem:P1},
where we encounter a matrix whose first row consists entirely of zeros. 

\begin{lem}\label{lem:P1}
Let 
\[A=\begin{bmatrix}A_1& A_2\end{bmatrix}\in T_0^{m\times(n_1+n_2)},\]
where $A_i\in T_0^{m\times n_i}$, $i=1,2$. 
Then there exist $P\in\P_m$ and $M\in\M_{n_2}$ such that
$PA_2M$ satisfies {\rm(S2)} and {\rm(S3)}$'$, 
and $\begin{bmatrix}PA_1 &PA_2M\end{bmatrix}$ satisfies {\rm(S4)}.
\end{lem}
\begin{proof}
Without loss of generality, we may assume 
$A_1$ satisfies (S4).
Then there exist  row vectors $\ba_1, \ldots, \ba_k$ of $A_1$ such that 
$\ba_1<\cdots<\ba_k$, and positive integers
$m_1,\dots,m_k$ such that
\[A_1=\begin{bmatrix}\allone^\top_{m_1}&&\\ &\ddots&
\\ &&\allone^\top_{m_k}\end{bmatrix}
\begin{bmatrix}\ba_1\\ \vdots\\ \ba_k\end{bmatrix},\]
where $\sum_{i=1}^k m_i=m$.
Write
\[A_2=\begin{bmatrix}B_1\\ \vdots\\ B_k\end{bmatrix},\]
where $B_i\in T_0^{m_i\times n_2}$ for $i=1,2,\dots,k$. 
We may assume $B_1\neq0$, since otherwise the proof reduces
to establishing the assertion for the matrix $A$ with
the first $m_1$ rows deleted.
Let $\bb$ be a row vector of $B_1$
with maximum number of
nonzero components. Then there exists $M\in\M_{n_2}$ such that
the vector $\bb M$ constitutes ones followed by zeros.
Moreover, for each 
$i\in\{1,\dots,k\}$, 
there exists 
$P_i\in\P_{m_i}$ such that the rows of $P_iB_iM$ are
in lexicographic order.
It follows that $\bb M$ is the first row of $P_1B_1M$, 
that is also the first row of $PA_2M$.
Set
$P=\diag(P_1,\dots,P_k)$.
Then $PA_2M$ satisfies (S3).
Since $PA_1=A_1$, 
we see that $\begin{bmatrix}PA_1 &PA_2M\end{bmatrix}$ satisfies (S4).

With the above notation, 
we prove the assertion by induction on $n_2$. First we treat the
case where $\bb M=\allone$. This in particular includes the case
where $n_2=1$, the starting point of the induction.
In this case, the first row of $PA_2M$ is $\allone$,
hence $PA_2M$ satisfies (S2). The other
assertions have been proved already.

Next we consider the case where
$\bb M=\begin{bmatrix}\allone_{n_2-n'_2}&\allz_{n'_2}\end{bmatrix}$,
with $0<n'_2<n_2$. 
Define 
$A'_1\in T_0^{m\times (n_1+n_2-n'_2)}$ and
$A'_2\in T_0^{m\times n'_2}$ by setting
$\begin{bmatrix} A'_1&A'_2\end{bmatrix}$ to be the
matrix obtained from
$\begin{bmatrix} A_1& PA_2M\end{bmatrix}$
by deleting the first row.
By inductive hypothesis, there exist $P'\in\P_{m-1}$ and $M'\in\M_{n'_2}$
such that $P'A'_2M'$ satisfies (S2) and (S3)$'$, 
and $\begin{bmatrix}P'A'_1 &P'A'_2M'\end{bmatrix}$ satisfies (S4).
By our choice of $\bb$, the row vector $\bb M$ is lexicographically
the smallest member among the rows of $P_1B_1M$, and the same is
true among the rows of the matrix $P_1B_1M''$, where
\[M''=M\begin{bmatrix}I_{n_2-n'_2}&0\\0&M'\end{bmatrix}.\]
It follows that the matrix
\[\begin{bmatrix} 1&0\\ 0&P'\end{bmatrix}
\begin{bmatrix} A_1& PA_2M''\end{bmatrix}
=\begin{bmatrix} \ast&0\\ P'A'_1&P'A'_2M'\end{bmatrix}\]
satisfies (S4).
Set
\[P''=\begin{bmatrix} 1&0\\ 0&P'\end{bmatrix}P.\]
Since $P'A'_2M'$ satisfies (S2), while the first row of $P''A_2M''$
is the same as that of $PA_2M$ which
is $\begin{bmatrix}\allone_{n_2-n'_2}&\allz_{n'_2}\end{bmatrix}$, 
the matrix $P''A_2M''$
satisfies both (S2) and (S3)$'$.
We have already shown that the matrix
$\begin{bmatrix} P''A_1 &P''A_2M\end{bmatrix}$
satisfies (S4).
\end{proof}

\begin{lem}\label{lem:P2}
Under the same assumption as in Lemma \ref{lem:P1}, there exist $M_1\in\M_{m}$ and $M_2\in\M_{n_2}$ such that $\begin{bmatrix}M_1A_1&M_1A_2M_2\end{bmatrix}$ satisfies {\rm(S1)} and {\rm(S4)}, 
and $M_1A_2M_2$ satisfies {\rm(S2)} and {\rm(S3)}$'$.  
\end{lem}
\begin{proof}
We will prove the assertion by induction on $m$. Suppose $m=1$. It is clear that every single row vector always satisfies (S4). Also, every single row vector satisfying (S3)$'$ necessarily satisfies (S2). Now, if $A_1=0$ or $n_1=0$, then there exists $M_2\in\M_{n_{2}}$ such that $A_2M_2$ satisfies (S3)$'$ and hence (S1) is satisfied.  If $A_1\neq 0$, then there exist $a\in T$ and $M_2\in\M_{n_2}$ such that $aA_1$ satisfies (S1) and $aA_2M_2$ satisfies (S3)$'$. 

Assume the assertion is true up to $m-1$. First, we consider the case where $A_1=0$ or $n_1=0$. Without loss of generality, we may assume $A_2\neq 0$. Furthermore, we may assume that the first row and the first column of $A_2$ are ones followed by zeros. Then there exists $P'\in\P_{n_2}$ such that 
\[
A_2P'=
\left[
\begin{array}{cc|c}
1&\begin{matrix}\allone&0\end{matrix}&0\\
\allone^T&\begin{matrix}B_1& B_2\end{matrix}&0\\
0&C_1&C_2
\end{array}
\right]
\]
where $B_2\in T_0^{m_1\times t}$ has no zero column. By Lemma \ref{lem:P1}, there exist $P\in\P_{m_1}$ and $M\in\M_{t}$ such that $PB_2M$ satisfies (S2) and (S3)$'$ and $\begin{bmatrix}PB_1&PB_2M\end{bmatrix}$ satisfies (S4). 
Let 
\[C_1'=C_1\begin{bmatrix}I_{n_2-n_2'-t-1}&0\\0&M\end{bmatrix}.\] 
By inductive hypothesis, there exist $M_1'\in\M_{m-m_1-1}$, and $M_2'\in\M_{n_2'}$ such that $\begin{bmatrix}M_1'C_1'&M_1'C_2M_2'\end{bmatrix}$ satisfies (S1) and (S4), and $M_1'C_2M_2'$ satisfies (S2) and (S3)$'$. By setting 
\[
M_1=\begin{bmatrix}1&0&0\\0&P&0\\0&0&M_1'\end{bmatrix},\quad M_2=P'\begin{bmatrix}I_{n_2-n_2'-t}&0&0\\0&M&0\\0&0&M_2'\end{bmatrix},
\]
the matrix $M_1A_2M_2$ satisfies (S1)--(S4).

Next we consider the case $A_1\neq 0$. Without loss of generality, we may assume that the first nonzero column in $A_1$ is ones followed by zeros. Write 
\[
A_1=
\begin{bmatrix}
0_{m\times t}&
\begin{array}{cc}
\bold{1}^{\T}&B_1\\
0&D_1
\end{array}
\end{bmatrix}
\]
for some $t<n_1$, with $B_1\in T_0^{m_1\times( n_1-t-1)}$ and $D_1\in T_0^{m_2\times (n_1-t-1)}$ for some $m_1,m_2$ with $m_1+m_2=m$ and $m_2<m$. 
Then there exists $P'\in\P_{n_2}$ such that \[
A_2P'=\begin{bmatrix}
B_2&0_{m_1\times n_2'}\\
D_2&C_2
\end{bmatrix}\] for some $n_2'\geq 0$, where $B_2\in T_0^{m_1\times (n_2-n_2')}$ has no zero column. By Lemma \ref{lem:P1}, there exist $P\in\P_{m_1}$ and $M\in\M_{n_2-n_2'}$ such that $PB_2M$ satisfies (S2) and (S3)$'$ and 
$\begin{bmatrix}PB_1&PB_2M\end{bmatrix}$ satisfies (S4). Let $C_1=\begin{bmatrix}D_1&D_2M\end{bmatrix}$. Then by inductive hypothesis, there exist $M_1'\in\M_{m_2}$ and $M_2'\in\M_{n_2'}$ such that $\begin{bmatrix}M_1'C_1&M_1'C_2M_2'\end{bmatrix}$ satisfies (S1) and (S4), and $M_1'C_2M_2'$ satisfies (S2) and (S3)$'$. By setting 
\[
M_1=\begin{bmatrix}P&0\\0&M_1'\end{bmatrix},\quad M_2=P'\begin{bmatrix}M&0\\0&M_2'\end{bmatrix},
\]
the proof is complete. 
\end{proof}

\begin{thm}\label{thm:1}
Every matrix in $T_0^{m\times n}$ is equivalent to a matrix that is in  standard form.
\end{thm}
\begin{proof}
Let $W\in T_0^{m\times n}$. Setting $A_1=\varnothing$ and $A_2=W$ in Lemma \ref{lem:P2}, we see that $W$ is equivalent to a matrix that is in standard form.
\end{proof}

\begin{cor}
Every unit weighing matrix is equivalent to a unit weighing matrix that is in standard form.
\end{cor}
\begin{proof}
Setting $T=\{z\in \mathbb{C}:|z|=1\}$, the proof is immediate from Theorem \ref{thm:1}.
\end{proof}

\section*{Acknowledgement}
The second author would like to acknowledge the Hitachi Global Foundation for providing a scholarship grant.

\end{document}